\documentclass[12pt]{amsart}
\usepackage{amsthm,amsmath,amsfonts,amssymb,latexsym,amscd,mathrsfs,hyperref,cleveref,graphicx}
\usepackage{setspace}
\setdisplayskipstretch{2}
\usepackage{cite}

\makeatletter
\renewcommand{\pod}[1]{\allowbreak\mathchoice
  {\if@display \mkern 18mu\else \mkern 8mu\fi (#1)}
  {\if@display \mkern 18mu\else \mkern 8mu\fi (#1)}
  {\mkern4mu(#1)}
  {\mkern4mu(#1)}
}

\theoremstyle{plain}

\newtheorem{lem}{Lemma}

\newtheorem*{thm*}{Theorem}
\newtheorem*{cor*}{Corollary}
\newtheorem*{prop*}{Proposition}

\theoremstyle{definition}

\newtheorem*{exa*}{Example}

\crefname{thm}{Theorem}{theorems}
\crefname{lem}{Lemma}{lemmas}

\newcommand{\mc}{\mathcal}

\newcommand{\re}{\operatorname{Re}}

  \def \d{\delta}     \def \s{\sigma}  \def \z{\zeta}

\numberwithin{equation}{section}

\renewcommand{\labelenumi}{\setlength{\labelwidth}{\leftmargin}
   \addtolength{\labelwidth}{-\labelsep}
   \hbox to \labelwidth{\theenumi.\hfill}}

\begin{document}
\title[A theorem of Bombieri-Vinogradov type]{A theorem of Bombieri-Vinogradov type with few exceptional moduli}
\author{Roger Baker}
\address{Department of Mathematics\newline
\indent Brigham Young University\newline
\indent Provo, UT 84602, U.S.A}
\email{baker@math.byu.edu}

 \begin{abstract}
Let $1 \le Q \le x^{9/40}$ and let $\mc S$ be a set of pairwise relatively prime integers in $[Q,2Q)$. The prime number theorem for arithmetic progressions in the form
 \[\max_{y\le x}\ \max_{\substack{a\\
 (a,q)=1}} \Bigg|\sum_{\substack{
 n \equiv a\pmod q\\
 n \le y}} \Lambda(n) - \frac x{\phi(q)}\Bigg| <
 \frac x{\phi(q)(\log x)^A}\]
holds for all $q$ in $\mc S$ with $O((\log x)^{34+A})$ exceptions.
 \end{abstract}
 


\maketitle

 \section{Introduction}\label{sec:intro}

Let $\Lambda(n)$ denote the von Mangoldt function. The prime number theorem in the form
 \[\sum_{\substack{
 n \le x\\
 n \equiv a\pmod q}} \Lambda(n) = \frac x{\phi(q)}
 \left(1 + O_A((\log x)^{-A})\right)\]
for every $A > 0$, holds for $q \le (\log x)^A$, $(a,q)=1$. The best-known average result for a set of moduli $q$ is the Bombieri-Vinogradov theorem. Let
 \begin{gather*}
E(x;q,a) = \sum_{\substack{
 n \le x\\
 n \equiv a\pmod q}} \Lambda(n)
 - \frac x{\phi(q)}\, ,\\[2mm]
 E(x,q) = \max_{\substack{
 a\\
 (a,q)=1}} |E(x; q, a)|\, , \, E^*(x,q) =
 \max_{y\le x} |E(y,q)|
 \end{gather*}
It is easy to deduce from the presentation of the Bombieri-Vinogradov theorem in \cite{dav} that 

 \[E^*(x,q) \le \frac x{\phi(q)(\log x)^A}\]
for all integers $q$ in $[Q, 2Q)$ with at most $O(Q(\log x)^{-A})$ exceptions, provided that $Q \le x^{1/2}(\log x)^{-2A-6}$.

It is of interest to restrict the size of this exceptional set further. Following Cui and Xue \cite{cuixue} we find that \textit{provided only prime moduli $q$ are considered}, the exceptional set has cardinality $O(\mc L^{C+A})$ for some absolute constant $C$ when $Q \le x^{1/5}$.

Glyn Harman has pointed out to me that one can obtain the result of \cite{cuixue} directly from Vaughan \cite[Theorem 1]{vau} with $C=3$.

In the present paper, the constant 1/5 is increased to 9/40 by adding a `Halasz-Montgomery-Huxley' bound to the tools employed in \cite{cuixue}; see \Cref{lem3} below. We also relax the primality condition a little.

 \begin{thm*}
Let $Q \le x^{9/40}$. Let $\mc S$ be a set of pairwise relatively prime integers in $[Q, 2Q)$. The number of $q$ in $\mc S$ for which
 \[E^*(x,q) > \frac x{\phi(q)(\log x)^A}\]
is $O((\log x)^{34+A})$.
 \end{thm*}
 
As a simple example, we may take $\mc S$ to be the set of prime powers in $[Q, 2Q)$. The constant 34 could be reduced with further effort. Constants implied by `$O$', `$\ll$' are absolute constants throughout the paper. We write $|\mc E|$ for the cardinality of a finite set $\mc E$ and
 \[\mc L = \log x.\]

We suppose, as we may, that $x$ is large.
 \bigskip

 \section{A proposition which implies the theorem}\label{sec:propimpliesthm}
 
We write
 \[\sideset{}{'}\sum_{\chi\pmod q} \ , \
 \sideset{}{^*}\sum_{\chi\pmod q}\]
for a sum respectively over non-principal characters and primitive characters $\pmod q$. For $y \le x$, let
 \[\psi(y,\chi) = \sum_{n\le y}
 \Lambda(n) \chi(n) \ , \ \psi_q(n) =
 \sum_{\substack{n\le y\\
 (n,q) =  1}} \Lambda(n).\]
We note the identity, for $(a,q)=1$,

 \begin{equation}\label{eq2.1}
\sum_{\substack{n\le y\\
n \equiv a\pmod q}} \Lambda(n) - \frac 1{\phi(q)}\, \psi_q(y) = \frac 1{\phi(q)}\ \sideset{}{'}\sum_{\chi\pmod q} \bar\chi(a) \psi(y,\chi). 
 \end{equation}

For brevity, we write $\d = 1/20$.

 \begin{prop*}
Let $Q \le x^{9/40}$. Then 
 \[S(Q) := \sum_{q < 2Q} \ \,
 \sideset{}{^*}\sum_{\chi \pmod q} \ \max_{y \le  x}
 |\psi(y, \chi)| \ll x\mc L^{34-\d}.\]
 \end{prop*}

The Proposition implies the Theorem. For if $q \in [Q,2Q)$, an argument on page 163 of \cite{dav} yields, for $y \le x$,
 \begin{equation}\label{eq2.2}
\frac 1{\phi(q)} \Bigg(\sum_{n \le y} \Lambda(n) - \psi_q(y)\Bigg) \ll \frac{\mc L^2 \log \mc L}Q 
 \end{equation}
and
 \begin{equation}\label{eq2.3}
\frac 1{\phi(q)}\, (\psi(y, \chi_1) - \psi(y,\chi)) \ll \frac{\mc L^2 \log \mc L}Q
 \end{equation}
where $\chi$ is induced by the primitive character $\chi_1$. Let
 \[E^\dag(x,q) = \max_{y \le x} \
 \max_{(a,q) = 1} \Bigg|\sum_{\substack{
 n \le y\\
 n \equiv a\pmod q}} \Lambda(n) - \frac 1{\phi(q)}\
 \sum_{n \le y} \Lambda(n)\Bigg|.\] 
We combine all contributions to $E^\dag(x,q)$ made by an individual primitive character. We see from \eqref{eq2.1}--\eqref{eq2.3} that
 \begin{align*}
\sum_{q \in  \mc S} E^\dag(x,Q) &\ll \sum_{q \le Q} \mc L^2 \log \mc L\\[2mm]
&\qquad + \sum_{3\le  q_1\, \le Q} \ \, \sideset{}{^*}\sum_{\chi_1\pmod{q_1}} \max_{y \le x}\, |\psi(y,\chi_1)| \ \sum_{\substack{
 \frac Q{q_1} \le k < \frac{2Q}{q_1}\\[1mm]
 q_1 k \in \mc S}}\frac 1{\phi(kq_1)}\\[2mm]
&\ll Q\mc L^3 + \frac{\log\mc L}Q \ \sum_{3 \le q_1 \le Q} \ \, \sideset{}{^*}\sum_{\chi_1\pmod{q_1}} \ \max_{y \le x} |\psi(y,\chi_1)| \sum_{\substack{
\frac Q{q_1} \le k < \frac{2Q}{q_1}\\
q_1 k \in \mc S}} 1.
 \end{align*}
The inner sum is 0 or 1 by our hypothesis on $\mc S$, and we obtain
 \[\sum_{q \in \mc S} E^\dag (x,Q) \ll Q\mc L^3 +
 \frac{\log \mc L}Q\, S(Q) \ll \frac{x\mc L^{34}}Q.\]

The set $\mc A$ of $q \in [Q,2Q)$ for which
 \[E^\dag (x,q) > \frac x{2\phi(q)}\, \mc L^{-A}\]
thus has cardinality
 \[|\mc A| \ll \mc L^{34+A}.\]
For $q \in \mc S_Q -\mc A$, $y \le x$, $(a,q)=1$ we have
 \begin{align*}
\Bigg|\sum_{\substack{n \le y\\
 n \equiv a\pmod q}} &\Lambda(n) - \frac y{\phi(q)}\Bigg|\\[2mm]
&\le \frac{x\mc L^{-A}}{2\phi(q)} +\Bigg|\sum_{n \le y} \Lambda(n) - \frac y{\phi(q)}\Bigg|\\[2mm] 
&\le \frac x{\phi(q)}\, \mc L^{-A}
 \end{align*}
by the prime number theorem. This completes the proof of the theorem.

We now explain the initial stage of the proof of the proposition. For $\chi$ $\pmod q$ a primitive character, $Q \le q < 2Q$, choose $y(\chi)$ to maximize
 \[\Bigg|\sum_{n \le y}
 \Lambda(n) \chi(n)\Bigg| \quad (y\le x)\]
and $a(\chi)$ so that $|a(\chi)|=1$,
 \[a(\chi) \sum_{n \le y(\chi)} \Lambda(n)
 \chi(n) = \Bigg|\sum_{n \le y(\chi)}
 \Lambda(n) \chi(n)\Bigg|.\]
Thus
 \[S(Q) = \sum_{q < 2Q}\ \,
 \sideset{}{^*}\sum_{\chi\pmod q} \, a(\chi)
 \sum_{n \le y(\chi)} \Lambda (n)\chi(n).\]
From the discussion in Heath-Brown \cite{hb}, $S(Q)$ is a linear combination, with bounded coefficients, of $O(\mc L^8)$ sums of the form
 \[S := \sum_{q < 2Q} \ \,
 \sideset{}{^*}\sum_{\chi\pmod q} a(\chi) \
 \sum_{\substack{n_1\ldots n_8 \le y(\chi)\\[1mm]
 n_i \in I_i}} (\log n_1) \mu(n_5)\ldots \mu(n_8)
 \chi(n_1)\ldots \chi(n_8)\]
in which $I_i = (N_i,2N_i]$, $\prod\limits_i N_i \le x$ and $2N_i \le x^{1/4}$ if $i>4$. Some of the intervals $I_i$ may contain only the integer 1, and we replace these by $[1,2)$ without affecting the upper bound $\prod\limits_i N_i \le x$. Now we need only bound $S$ by $O(x\mc L^{26-\d}/Q)$.

It is convenient to get rid of the factor $\log n_1$ in $S$. We have
 \[S =\sum_{q,\chi} \int_1^{N_1} \frac 1v \
 \sum_{\substack{
 n_i \in I_i'\\[1mm]
 n_1\ldots n_8 \le y(\chi)}} \mu(n_5) \ldots \mu(n_8)\
 \chi(n_1) \ldots \chi(n_8)\, dv,\]
where $I_1 = (\max(v, N_1), 2N_1]$ and $I_i'= I_i$ for $i>1$.

Next we use Perron's formula \cite[Lemma 3.12]{titch}. Let
 \[F_j(s,\chi) = F_j(s, \chi, v) = \sum_{n\in I_j'}\
 a_j(n) \chi(n) n^{-s}\]
where $a_j(n) = 1$ $(j\le 4)$, $a_j(n) = \mu(n)$ $(j > 4)$. Then

 \begin{align*}
&\sum_{\substack{
n_i\in I_i'\\
n_1\ldots n_8 \le y(\chi)}} \mu(n_5)\ldots\mu(n_8)\chi(n_1)\ldots \chi(n_8)\\[2mm]
&\qquad =\frac 1{2\pi i} \int_{1+\mc L^{-1}-ix}^{1+\mc L^{-1}+ix} F_1(s,\chi)\ldots F_8(s,\chi)\, \frac{y(\chi)^s}s\, dx + O(\mc L^2).
 \end{align*}

We shift the path of integration to $\re(s) = 1/2$. We have
 \[\left|\prod_{j=1}^8 F_j(\s \pm ix, \chi)\right| \le
 \prod_{j=1}^8 N_j^{1-\s} \le x^{1-\s},\]
so that the integral on the horizontal segments is $O(1)$. Thus
 \begin{align*}
S &= \frac 1{2\pi i} \ \sum_{q,\chi} \int_1^{N_1} \frac 1v \int_{-x}^x\ \prod_{j=1}^8 F_j\left(\frac 12 + it, \chi\right) \ \frac{y(\chi)^{\frac 12+it}}{\frac 12 + it}\, dt\, dv  + O(Q^2\mc L^2)\\[2mm]
&\ll x^{1/2}\mc L \sum_{q,\chi,T} T^{-1} \int_{-T}^T \ \prod_{j=1}^8\ \left|F_j\left(\frac 12  + it, \chi\right)\right|dt + O(Q^2\mc L^2) 
 \end{align*}
where $T$ takes the values $2^k$, $1 \le k \le \mc L/\log 2$. Here $v$ is now fixed in $[1,N_1]$. Since $Q^4 <x$, we need only show that
 \begin{align}
\sum_{q < 2Q}\ \ \sideset{}{^*}\sum_{\chi\pmod q} \int_{-T}^T \ \prod_{j=1}^8 \ & \left|F_j\left(\frac 12 + t, \chi\right)\right|dt\label{eq2.4}\\[2mm]
& \ll T^{39/40}x^{1/2}\mc L^{25-\d}\ (1 \le T \le x). \notag
 \end{align}
 
This is done by grouping $F_1\ldots F_8$ into two or three subproducts. It is time to state the lemmas we need on Dirichlet polynomials. For the rest of this section, let
 \begin{equation}\label{eq2.5}
S(s,\chi) = \sum_{n=N}^{N'} a_n \chi(n)n^{-s}
 \end{equation}
where $1\le N \le x$, $N \le N' \le cN$ with an absolute constant $c$, and let
 \[G = \sum_{n=N}^{N'} |a_n|^2.\]

 \begin{lem}\label{lem1}
Let $1 \le T$, $Q\le x$. For a primitive character $\chi\pmod q$ let $J_\chi$ be a set of numbers in $[-T,T]$ such that $|t-t'| \ge 1$ for distinct $t$, $t'$ in $J_\chi$. Then
 \begin{equation}\label{eq2.6}
\sum_{q < 2Q} \ \sideset{}{^*}\sum_{\chi\pmod q}\ \sum_{t\in J_\chi} |S(it, \chi)|^2 \ll \mc L(Q^2T + N)G.
 \end{equation}
 \end{lem}
 
 \begin{proof}
This follows at once from Theorem 7.3 of Montgomery \cite{mont}.
 \end{proof}
 
 \begin{lem}\label{lem2}
Let $a_n=1$ $(N \le n \le N')$ in \eqref{eq2.4}. Let $J_\chi$ be as in \Cref{lem1}. Then
 \begin{equation}\label{eq2.7}
\sum_{q < 2Q}\ \, \sideset{}{^*}\sum_{\chi\pmod q}\ \sum_{t \in J_\chi}\ \left|S\left(\frac 12 + it, \chi\right)\right|^4 \ll Q^2 T\mc L^{10}.
 \end{equation}
 \end{lem}
 
 \begin{proof}
Following the argument of Liu and Liu \cite{liuliu}, proof of Proposition 5.3, we find that
 \begin{equation}\label{eq2.8}
M_1: = \sum_{q < 2Q}\ \, \sideset{}{^*}\sum_{\chi \pmod q} \ \int_{-T}^T\ \left|S\left(\frac 12 + it, \chi\right)\right|^4 dt \ll Q^2T\mc L^9
 \end{equation}
and, for the derivative $S'$,
 \begin{equation}\label{eq2.9}
M_2: = \sum_{q < 2Q}\ \, \sideset{}{^*}\sum_{\chi\pmod q} \int_{-T}^T \ \left|S'\left(\frac 12 + it, \chi\right)\right|^2 dt \ll Q^2 T\mc L^{13}.
 \end{equation}
We now appeal to Lemma 1.4 of \cite{iwa} with $S^2$, $2SS'$ in place of $S'$ This gives for the left-hand side of \eqref{eq2.7} the bound
 \begin{align*}
&\sum_{q< 2Q}\ \, \sideset{}{^*}\sum_{\chi\pmod q} \Bigg\{\int_{-T}^T \left|S\left(\frac 12 + it, \chi\right)\right|^4 dt\\[2mm]
& + \left(\int_{-T}^T\ \left|S\left(\frac 12 + it, \chi\right)\right|^4 dt\right)^{1/2} \left(\int_{-T}^T \left|2S\left(\frac 12 + it, \chi\right) S'\left(\frac 12 + it,\chi\right)\right|^2dt\right)\Bigg\}.
 \end{align*}
By the Cauchy-Schwarz inequality, the product contributes at most
 \[\sum_{q< 2Q} \ \, \sum_{\chi\pmod q} \left(
 \int_{-T}^T\, \left|S\left(\frac 12 + it, \chi\right)
 \right|^4dt\right)^{3/4} \left(\int_{-T}^T\,
 \left|S'\left(\frac 12 + it,\chi\right)\right|^4
 dt\right)^{1/4},\]
which by H\"older's inequality is at most $M_1^{3/4}M_2^{1/4}$. The proof is now completed using \eqref{eq2.8}, \eqref{eq2.9}.
 \end{proof}
 \bigskip 
 
 \begin{lem}\label{lem3}
Let $\mc B$ be the set of $(q, \chi,t)$ with $q \le Q$, $\chi\pmod q$, $t\in J_\chi$ and
 \[|S(it, \chi)| \ge V > 0\]
in \Cref{lem1}. Then
 \begin{equation}\label{eq2.10}
|\mc B| \ll GNV^{-2}\mc L^6 + G^3NQ^2TV^{-6}\mc L^{18}.
 \end{equation}
 \end{lem}
 
 \begin{proof}
This is a very slight variant of Iwaniec and Kowalski \cite[Theorem 9.18]{iwakow}.
 \end{proof}

Let $\tau_b(n)$ be the number of factorizations $n=n_1\ldots n_b$.

If $S(it, \chi)$ is the product of $b$ of the above functions $F_j\left(\frac 12 +  it, \chi\right)$, it is clear that
 \begin{equation}\label{eq2.11}
G \le N^{-1} \sum_{n \le cN} \tau_b^2(n) \ll \mc L^{b^2-1}.
 \end{equation}
The last step is a standard application of Perron's formula to
 \[\sum_{n=1}^\infty \ 
 \frac{\tau_b^2(n)}{n^s}\, ,\]
which we can write as $F(s)\z(s)^{b^2}$ with $F$ analytic and bounded in $\re(s) \ge 2/3$. It follows that, with $J_\chi$ as in \Cref{lem1},
 \begin{equation}\label{eq2.12}
\sum_{q <2Q}\ \, \sideset{}{^*}\sum_{\chi\pmod q} \ \, \sum_{t\in J_\chi} \ \left|F_j\left(\frac 12 + it, \chi\right)\right|^2 \ll x^{9/20} T\mc L^{10} 
 \end{equation}
for $N_j \ll x^{9/40}$ (using \Cref{lem1} and \eqref{eq2.11}) and for $N_j > x^{1/4}$ (using \Cref{lem2}). This explains the role of the `difficult interval' $(9/40, 1/4)$ in \Cref{lem4} below.
 \smallskip
 
 \section{Proof of the Proposition}\label{sec:proofofprop}

 \begin{lem}\label{lem4}
Let $u_1 \ge \cdots \ge u_8 \ge 0$, $u_1 + \cdots + u_8 \le 1$. Then either
 \medskip
 
(a) there is a partition $\{i\}$, $\mc A_1$, $\mc A_2$ of $\{1, \ldots, 8\}$ with $\max(|\mc A_1|), |\mc A_2|) \le 5$,

 \[u_i \not\in (9/40, 1/4) \ , \ \max
 \Bigg(\sum_{j\in \mc A_1}\, u_j\ , \ \sum_{j\in \mc A_2}
 \, u_j\Bigg) \le 9/20,\]
or
 \medskip
 
(b) there is a partition $\mc A_1$, $\mc A_2$ of $\{1, \ldots, 8\}$ with $\max(|\mc A_1|, |\mc A_2|) \le 6$,
 \[\sum_{j\in\mc A_i}\, u_j \le 11/20 \quad (i=1,2).\]
 \end{lem}
 
 \begin{proof}
If $u_1 + \cdots + u_5 \le 11/20$ we have (b) with $\mc A_1 = \{1, \ldots, 5\}$ since $u_6 + u_7 + u_8 \le \frac 38$. Assume $u_1 + \cdots + u_5 > 11/20$. Let $k$ be the least integer such that
 \[u_1 + \cdots + u_k \ge 9/20.\]
One of the following cases must occur.
 \begin{enumerate}
\item[(i)] $u_1 \not\in (9/40, 1/4)$  , $u_2 + u_4 + u_6 + u_8 > 9/20$. 
 \bigskip

\item[(ii)]  $u_1 \not\in (9/40, 1/4)$  , $u_2 + u_4 + u_6 + u_8 \le 9/20$.
 \bigskip

\item[(iii)]  $u_1 \in (9/40, 1/4)$  , $u_1 + \cdots + u_k \le 11/20$.
 \bigskip

\item[(iv)]  $u_1 \in (9/40, 1/4)$  , $u_1 + \cdots + u_k > 11/20$.
 \end{enumerate}
 \bigskip

In Case (i) we have
 \[9/20 < u_2 + u_4 + u_6 + u_8 \le 1/2\]
and (b) holds with $\mc A_1 = \{2, 4, 6, 8\}$.
 \bigskip

In Case (ii) we have $u_3 + u_5 +u_7 \le u_2 + u_4 + u_6$,and (a) holds with $i=1$, $\mc A_1 = \{3, 5,7\}$.
 \bigskip

In Case (iii), (b) holds with $\mc A_1 = \{1, \ldots, k\}$.
 \bigskip

In Case (iv), we have $k\ge3$. Now (a) holds with $i=2$, $\mc A_1 = \{1, 3, \ldots, k\}$. For $u_2 \le \frac{u_1+u_2}2 < 9/40$ and $u_1 + u_3 + \cdots + u_k \le u_1 + u_2 + \cdots + u_{k-1} < 9/20$.
 \end{proof}
 
 \begin{proof}[Proof of the Proposition]
In place of \eqref{eq2.4}, it clearly suffices to show that, with $J_\chi$ as in \Cref{lem1},

 \begin{align}
E := \sum_{q < 2Q}\ \, \sideset{}{^*}\sum_{\chi\pmod q} \ \, \sum_{t \in J_\chi}\ &\prod_{j=1}^8 \left|F_j\left(\frac 12 + it, \chi\right)\right|\label{eq3.1}\\[4mm]
& \ll T^{39/40} x^{1/2} \mc L^{25-\d}.\notag
 \end{align}

We reorder $N_1,\ldots, N_8$ so that $N_1 \ge \cdots \ge N_8$ and write $N_j = x^{u_j}$ with $u_1 \ge \cdots \ge u_8 \ge 0$, $u_1 + \cdots + u_8 \le 1$. Suppose we are in Case (b) of \Cref{lem4}. Let us write

 \begin{equation}\label{eq3.2}
\prod_{j\in \mc A_1} N_j = M, \ \prod_{j\in \mc A_2} N_j = N, \ S_\ell(it, \chi) = \prod_{j\in \mc A_\ell} F_j\left(\frac 12 + it, \chi\right), b = |\mc A_2|.
 \end{equation}
We bound $E$ using Cauchy's inequality and \eqref{eq2.6}, \eqref{eq2.11} for $S= S_1, S_2$:
 \begin{align*}
E &\ll \left((Q^2T + M) \mc L^{(8-b)^2}\right)^{1/2} \left((Q^2T + N) \mc L^{b^2}\right)^{1/2}\\[4mm]
&\ll \left(Q^2T + x^{1/2} + (Q^2T)^{1/2}(\max(M,N))^{1/2}\right) \mc L^{20}.
 \end{align*}
Now
 \begin{gather*}
Q^2T \ll T^{19/20} x^{1/2}\\[2mm]
(Q^2T)^{1/2}(\max (M,N))^{1/2} \ll T^{1/2} x^{9/40+11/40}\ll T^{1/2} x^{1/2}.
 \end{gather*}
This is acceptable in \eqref{eq3.1}

Suppose now we are in Case (a) of \Cref{lem4}. The argument mimics one due to Iwaniec \cite{iwa}. We retain the notation \eqref{eq3.2}, and write
 \[L = x^{u_i}.\]
The contribution to $E$ from those $t$ with
 \[\min\left(\left|F_i\left(\frac 12 + it, \chi\right)
 \right|, \ |S_1(it,\chi)|, \ |S_2(it, \chi)|\right) 
 \le x^{-1}\]
is at most
 \[Q^2T x\, x^{-1} \ll T^{39/40}
 x^{1/2} \mc L^{25-\d}.\]
By a simple splitting-up argument, there is a subset $\mc B$ of the set of triples $(q, \chi,t)$, $q < 2Q$, $\chi\pmod q$, $t\in J_\chi$ in \eqref{eq3.1} such that, for $(q, \chi, t)\in \mc B$, we have
 \begin{align*}
U &\le \left|F_i\left(\frac 12 + it, \chi\right)\right| < 2U,\\[2mm]
V& \le |S_1(it, \chi)| < 2V\\[2mm]
W &\le |S_2(it, \chi)| < 2W
 \end{align*}
for positive numbers $U$, $V$, $W$ with
 \[x^{-1} \le U \ll x^{u_i/2},
 x^{-1} \le V \ll M^{1/2}, \ x^{-1}
 \le W \ll N^{1/2},\]
while
 \begin{align}
E &\ll UVW |\mc B| \mc L^3\label{eq3.3}\\[2mm]
&\ll UVWP\mc L^3.\notag
 \end{align}
Here
 \begin{align*}
P &= \min \bigg(\frac{(M + x^{9/20}T)\mc L^{(7-b)^2}}{V^2}, \ \frac{(N + x^{9/20}T)\mc L^{b^2}}{W^2}, \ \frac{x^{9/20}T\mc L^{10}}{U^4},\\[2mm]
&\frac M{V^2}\ \mc L^{(7-b)^2+5} + \frac{Mx^{9/20}T}{V^6}\, \mc L^{3(7-b)^2+15}, \ \frac N{W^2}\, \mc L^{b^2+5} + \frac{x^{9/20}T\mc L^{3b^2+15}}{W^6},\\[2mm]
& \frac{L^2}{U^4}\, \mc L^9 + \frac{L^2x^{9/20}T}{U^{12}}\, \mc L^{27}\bigg),
 \end{align*}
and we have used \eqref{eq2.6}, \eqref{eq2.7}, \eqref{eq2.10}, \eqref{eq2.11} in the second step in \eqref{eq3.3}. It remains to show that
 \begin{equation}\label{eq3.4}
UVWP \ll T^{39/40} x^{1/2} \mc L^{22-\d}.
 \end{equation}
We consider four cases.
 \medskip
 
\noindent\textit{Case 1.} $P \le \dfrac{2M}{V^2}\, \mc L^{(7-b)^2+5}$, $P \le \dfrac{2N}{W^2}\, \mc L^{b^2+5}$.
 \medskip

\noindent In this case
 \begin{align*}
UVWP &\le 2UVW \min\left(V^{-2}M\mc L^{(7-b)^2+5}, W^{-2} N\mc L^{b^2+5}\right)\\[2mm]
&\ll U(MN)^{1/2} \mc L^{\frac 12((7-b)^2+b^2+10)} \ll x^{1/2} \mc L^{20}. 
 \end{align*}

\noindent\textit{Case 2.} $P > \dfrac{2M}{V^2}$ $\mc L^{(7-b)^2+5}$, $P > 2\, \dfrac N{W^2}$ $\mc L^{b^2+5}$.
In this case,
 \[P \le 2A_1 + 2B_1,\]
where
 \begin{align*}
A_1 &= \min(x^{9/20} TV^{-2} \mc L^{(7-b)^2+5}, \ x^{9/20}TW^{-2}\mc L^{b^2+5},\\[2mm]
&\qquad x^{9/20}TMV^{-6} \mc L^{3(7-b)^2+15}, \ x^{9/20} TNW^{-6}\mc L^{3b^2+15},\\[2mm]
&\hskip 1.5in x^{9/20} TU^{-4}\mc L^{10}, \ L^2 U^{-4}\mc L^9)
 \end{align*}
and
 \begin{align*}
B_1 &= \min(x^{9/20}TV^{-2} \mc L^{(7-b)^2+5}, \ x^{9/20} TW^{-2} \mc L^{b^2+5},\\[2mm]
&\qquad x^{9/20} TMV^{-6}\mc L^{3(7-b)^2+15}, \ x^{9/20} TNW^{-6} \mc L^{3b^2+15},\\[2mm]
&\hskip 1.5in x^{9/20} TU^{-4} \mc L^{10}, \ x^{9/20} TL^2 U^{-12}\mc L^{27}).
 \end{align*}
We have, for a constant $K_1$,
 \begin{align*}
A_1 &\le \mc L^{K_1}(x^{9/20}TV^{-2})^{5/16}(x^{9/20} TW^{-2})^{5/16}(x^{9/20}TMV^{-6})^{1/16}\cdot\\[2mm]
&\hskip 1in (x^{9/20}TW^{-6})^{1/16} \min(x^{9/20}TU^{-4}, L^2 U^{-4})^{1/4}\\[2mm]
&\ll \mc L^{K_1}(UVW)^{-1}T x^{9/20}(MN)^{1/16} \min(1, x^{-9/80} T^{-1/4} L^{1/2}).
 \end{align*}
We bound the last minimum by $(x^{-9/80}T^{-1/4}L^{1/2})^{1/8}$, obtaining
 \[A_1 \ll \mc L^{K_1}(UVW)^{-1} T^{31/32} x^{319/640},\]
which is acceptable in \eqref{eq3.4}. Now
 \begin{align*}
B_1 &\le \min ((x^{9/20}TV^{-2}\mc L^{(7-b)^2+5})^{5/16}(x^{9/20} TW^{-2} \mc L^{b^2+5})^{5/16}\, \cdot\\[2mm]
&\quad (x^{9/20}TMV^{-6}\mc L^{3(7-b)^2+15})^{1/16} (x^{9/20}TNW^{-6} \mc L^{3b^2+15})^{1/16} (x^{9/20} U^{-4} \mc L^{10})^{1/4},\\[2mm]
&\quad (x^{9/20}TV^{-2} \mc L^{(7-b)^2+5})^{7/16} (x^{9/20} TW^{-2} \mc L^{b^2+5})^{7/16} (x^{9/20}TMV^{-6} \mc L^{3(7-b)^2+15})^{1/48}\, \cdot\\[2mm]
&\quad (x^{9/20} TNW^{-6} \mc L^{3b^2+15})^{1/48}(x^{9/20} TL^2 U^{-12} \mc L^{27})^{1/12})\\[2mm]
&\ll (UVW)^{-1} x^{9/20} T^{3/4}(MN)^{1/16} \min(\mc L^{K_2}, T^{1/4} L^{1/6}(MN)^{-1/24} \mc L^{K_3})
 \end{align*}
where
 \begin{align*}
K_2 &= \left((7-b)^2 + b^2\right) \left(\frac 5{16} + \frac 3{16}\right) + \frac{50}{16} + \frac{30}{16} + \frac{10}4 \le 22,\\[2mm]
K_3 &= \left((7-b)^2 + b^2\right) \left(\frac 7{16} + \frac 3{48}\right) + \frac{70}{16} + \frac{30}{48} + \frac{27}{12} \le 22 - \frac 14.
 \end{align*}
We bound the last minimum by $\mc L^{22-3/40}(T^{1/4}L^{1/6}(MN)^{-1/24})^{3/10}$, obtaining
 \[B_1 \ll (UVW)^{-1} T^{39/40} x^{1/2}
 \mc L^{22-\d},\]
which is acceptable in \eqref{eq3.4}.
 \bigskip

\noindent\textit{Case 3.} $P > 2V^{-2} M\mc L^{(7-b)^2+5}$, $P \le 2W^{-2}N\mc L^{b^2+5}$. In this case, for a constant $K_4$,
 \[P \le \mc L^{K_4}(A_2 + B_2)\]
where
 \begin{align*}
A_2 &= \min(x^{9/20} TV^{-2}, NW^{-2}, x^{9/20} TMV^{-6}, x^{9/20} TU^{-4}, L^2 U^{-4}),\\[2mm]
B_2 &= \min(x^{9/20} TV^{-2}, NW^{-2}, x^{9/20} TMV^{-6}, x^{9/20} TU^{-4}, x^{9/20} T L^2 U^{-12}).
 \end{align*}
Now
 \begin{align*}
A_2 &\le (x^{9/20} TV^{-2})^{1/8}(NW^{-2})^{1/2}(x^{9/20} TMV^{-6})^{1/8}\\[2mm]
&\hskip .75in \min (x^{9/20} TU^{-4}, L^2 U^{-4})^{1/4}\\[2mm]
&= (UVW)^{-1} (x^{9/20} TN)^{1/2} M^{1/8} \min(1, x^{-9/80} T^{-1/4} L^{1/2}).
 \end{align*}
We bound the last minimum by $(x^{-9/80} T^{-1/4} L^{1/2})^{1/4}$, obtaining
 \[A_2 \ll (UVW)^{-1} T^{7/16} x^{157/320}\]
(using $N \le x^{9/20}$), which is acceptable. Further,
 \begin{align*}
B_2 &\ll \min((x^{9/20} TV^{-2})^{1/8}(NW^{-2})^{1/2}(x^{9/20} TMV^{-6})^{1/8}(x^{9/20} TU^{-4})^{1/4},\\[2mm]
&\qquad (x^{9/20} TV^{-2})^{3/8} (NW^{-2})^{1/2} (x^{9/20} TMV^{-6})^{1/24} (x^{9/20} L^2 U^{-12})^{1/12})\\[2mm]
&= (UVW)^{-1}(x^{9/20} TN)^{1/2} M^{1/8} \min(1, L^{1/6} M^{-1/12}).
 \end{align*}
We bound the last minimum by $(L^{1/6} M^{-1/12})^{1/2}$. Similarly to the bound for $A_2$,
 \[B_2 \ll (UVW)^{-1} T^{1/2} x^{119/240},\]
which is acceptable.
 \bigskip

\noindent\textit{Case 4.} $P \le 2V^{-2} M \mc L^{(7-b)^2+5}$, $P > 2W^{-2} N\mc L^{b^2+5}$. We proceed as in Case 3, interchanging the roles of $S_1$ and $S_2$.
 \medskip

This establishes \eqref{eq3.4} and completes the proof of the Proposition.
 \end{proof}

\hskip .5in


\begin{thebibliography}{www}

\bibitem{cuixue} Z.~Cui and B.~Xue, A note on the distribution of primes in arithmetic progressions, \textit{Number Theory--Arithmetic in Shangri-La}, 83--89, World Sci. Publ. Hackensack, NJ, 2013.

\bibitem{dav} H.~Davenport, \textit{Multiplicative Number Theory}, 2nd edn., Springer, Berlin, 1980.

\bibitem{hb} D.~R.~Heath-Brown, Prime numbers in short intervals and a generalized Vaughan identity, \textit{Can. J. Math}. \textbf{34} (1982), 1365--1377.

\bibitem{iwa} H.~Iwaniec, On the Brun-Titchmarsh theorem, \textit{J. Math. Soc. Japan} \textbf{34} (1982), 95--123.

\bibitem{iwakow} H.~Iwaniec and E.~Kowalski, \textit{Analytic Number Theory}, American Math. Soc., Providence, RI, 2004.

\bibitem{liuliu} J.~Liu and M-C.~Liu, The exceptional set in the four prime squares problem, \textit{Illinois J. Math.} \textbf{44} (2000), 272--293.

\bibitem{mont} H.~L.~Montgomery, \textit{Topics in Multiplicative Number Theory}, Springer, Berlin, 1971.

\bibitem{titch} E.~C.~Titchmarsh, \textit{The Theory of the Riemann Zeta-Function}, 2nd edn., Oxford University Press, Oxford, 1986.

\bibitem{vau} R.~C.~Vaughan, Mean value theorems in prime number theory, \textit{J. London Math. Soc.} \textbf{10} (1975), 153--162.

 \end{thebibliography}
 \end{document}